\documentclass[12pt]{amsart}
 \usepackage{amsmath,amssymb,amsbsy,amsfonts,amsthm,latexsym,mathabx,
            amsopn,amstext,amsxtra,euscript,amscd,stmaryrd,mathrsfs,
            cite,array,mathtools,color,bm}
\usepackage{url}
\usepackage[colorlinks,linkcolor=blue,anchorcolor=blue,citecolor=blue,backref=page]{hyperref}
\usepackage{color}
\usepackage{float}

\hypersetup{breaklinks=true}
\usepackage{multirow}

\usepackage{algorithm}
\usepackage{algorithmic}
\usepackage{pgf}
\usepackage{tikz}
\usetikzlibrary{arrows,automata}

\begin{document}

\newtheorem{theorem}{Theorem}
\newtheorem{lemma}[theorem]{Lemma}
\newtheorem{claim}[theorem]{Claim}
\newtheorem{cor}[theorem]{Corollary}
\newtheorem{conj}[theorem]{Conjecture}
\newtheorem{prop}[theorem]{Proposition}
\newtheorem{definition}{Definition}
\newtheorem{question}[theorem]{Question}
\newcommand{\hh}{{{\mathrm h}}}

\numberwithin{equation}{section}
\numberwithin{theorem}{section}
\numberwithin{algorithm}{section}
\numberwithin{table}{section}
\numberwithin{figure}{section}

\renewcommand*{\backref}[1]{}
\renewcommand*{\backrefalt}[4]{%
    \ifcase #1 (Not cited.)%
    \or        (p.\,#2)%
    \else      (pp.\,#2)%
    \fi}

\def\sssum{\mathop{\sum\!\sum\!\sum}}
\def\ssum{\mathop{\sum\ldots \sum}}
\def\iint{\mathop{\int\ldots \int}}

\def\squareforqed{\hbox{\rlap{$\sqcap$}$\sqcup$}}
\def\qed{\ifmmode\squareforqed\else{\unskip\nobreak\hfil
\penalty50\hskip1em\null\nobreak\hfil\squareforqed
\parfillskip=0pt\finalhyphendemerits=0\endgraf}\fi}


\newcommand{\bflambda}{{\boldsymbol{\lambda}}}
\newcommand{\bfmu}{{\boldsymbol{\mu}}}
\newcommand{\bfxi}{{\boldsymbol{\xi}}}
\newcommand{\bfrho}{{\boldsymbol{\rho}}}

\def\fK{\mathfrak K}
\def\fT{\mathfrak{T}}

\def\fA{{\mathfrak A}}
\def\fB{{\mathfrak B}}
\def\fC{{\mathfrak C}}

\def \balpha{\bm{\alpha}}
\def \bbeta{\bm{\beta}}
\def \bgamma{\bm{\gamma}}
\def \blambda{\bm{\lambda}}
\def \bchi{\bm{\chi}}
\def \bphi{\bm{\varphi}}
\def \bpsi{\bm{\psi}}

\def\eqref#1{(\ref{#1})}

\def\vec#1{\mathbf{#1}}


\def\cA{{\mathcal A}}
\def\cB{{\mathcal B}}
\def\cC{{\mathcal C}}
\def\cD{{\mathcal D}}
\def\cE{{\mathcal E}}
\def\cF{{\mathcal F}}
\def\cG{{\mathcal G}}
\def\cH{{\mathcal H}}
\def\cI{{\mathcal I}}
\def\cJ{{\mathcal J}}
\def\cK{{\mathcal K}}
\def\cL{{\mathcal L}}
\def\cM{{\mathcal M}}
\def\cN{{\mathcal N}}
\def\cO{{\mathcal O}}
\def\cP{{\mathcal P}}
\def\cQ{{\mathcal Q}}
\def\cR{{\mathcal R}}
\def\cS{{\mathcal S}}
\def\cT{{\mathcal T}}
\def\cU{{\mathcal U}}
\def\cV{{\mathcal V}}
\def\cW{{\mathcal W}}
\def\cX{{\mathcal X}}
\def\cY{{\mathcal Y}}
\def\cZ{{\mathcal Z}}
\newcommand{\rmod}[1]{\: \mbox{mod} \: #1}

\newcommand{\sA}{\ensuremath{\mathscr{A}}}
\newcommand{\sB}{\ensuremath{\mathscr{B}}}
\newcommand{\sC}{\ensuremath{\mathscr{C}}}
\newcommand{\sG}{\ensuremath{\mathscr{G}}}
\newcommand{\sH}{\ensuremath{\mathscr{H}}}

\def\cg{{\mathcal g}}

\def\vr{\mathbf r}

\def\e{{\mathbf{\,e}}}
\def\ep{{\mathbf{\,e}}_p}
\def\er{{\mathbf{\,e}}_r}

\def\Tr{{\mathrm{Tr}}}
\def\Nm{{\mathrm{Nm}}}

\def\bc{{\mathbf{c}}}
\def\bC{{\mathbf{C}}}

 \def\SS{{\mathbf{S}}}

\def\lcm{{\mathrm{lcm}}}

\def\({\left(}
\def\){\right)}
\def\fl#1{\left\lfloor#1\right\rfloor}
\def\rf#1{\left\lceil#1\right\rceil}

\def\mand{\quad \mbox{and} \quad}

\definecolor{olive}{rgb}{0.3, 0.4, .1}
\definecolor{dgreen}{rgb}{0.,0.6,0.}

\newcommand{\commB}[1]{\marginpar{%
\begin{color}{red}
\vskip-\baselineskip 
\raggedright\footnotesize
\itshape\hrule \smallskip B: #1\par\smallskip\hrule\end{color}}}

\newcommand{\commD}[1]{\marginpar{%
\begin{color}{dgreen}
\vskip-\baselineskip 
\raggedright\footnotesize
\itshape\hrule \smallskip D: #1\par\smallskip\hrule\end{color}}}

\newcommand{\commI}[1]{\marginpar{%
\begin{color}{blue}
\vskip-\baselineskip 
\raggedright\footnotesize
\itshape\hrule \smallskip I: #1\par\smallskip\hrule\end{color}}}

\newcommand{\commM}[1]{\marginpar{%
\begin{color}{magenta}
\vskip-\baselineskip 
\raggedright\footnotesize
\itshape\hrule \smallskip M: #1\par\smallskip\hrule\end{color}}}




\hyphenation{re-pub-lished}

\mathsurround=1pt

\def\bfdefault{b}
\overfullrule=5pt

\def \bG{{\mathbf G}}

\def \F{{\mathbb F}}
\def \K{{\mathbb K}}
\def \Z{{\mathbb Z}}
\def \Q{{\mathbb Q}}
\def \R{{\mathbb R}}
\def \C{{\\mathbb C}}
\def\Fp{\F_p}
\def \fp{\Fp^*}

\def\Kmn{\cK_p(m,n)}
\def\psmn{\psi_p(m,n)}
\def\SpAB{\cS_{a,p}(\cA,\cB;\cI,\cJ)}
\def\SrAB{\cS_{a,r}(\cA,\cB;\cI,\cJ)}
\def\SpkAB{\cS_{a,p^k}(\cA,\cB;\cI,\cJ)}

\def\SpA{\cS_{a,p}(\cA;\cI,\cJ)}
\def\SrA{\cS_{a,r}(\cA;\cI,\cJ)}
\def\SpkA{\cS_{a,p^k}(\cA,\cI,\cJ)}

\def\Sp{\cS_{a,p}(\cI,\cJ)}
\def\Sr{\cS_{a,r}(\cI,\cJ)}
\def\Spk{\cS_{a,p^k}(\cI,\cJ)}

\def\RpIJ{R_{a,p}(\cI,\cJ)}
\def\RrIJ{R_{a,r}(\cI,\cJ)}
\def\RpkIJ{R_{a,p^k}(\cI,\cJ)}

\def\TpIJ{T_{a,p}(\cI,\cJ)}
\def\TrIJ{T_{a,r}(\cI,\cJ)}
\def\TpkIJ{T_{a,p^k}(\cI,\cJ)}
\def \xbar{\overline x_p}

\title[Functional Graphs of Quadratic Polynomials]{On Functional Graphs of Quadratic Polynomials}

\author[B. Mans]{Bernard Mans}
\address{B.M.: Department of Computing, Macquarie University, Sydney, NSW 2109, Australia}
\email{bernard.mans@mq.edu.au}

\author[M. Sha]{Min Sha}
\address{M.S.: Department of Computing, Macquarie University, Sydney, NSW 2109, Australia}
\email{shamin2010@gmail.com}

\author[I.~E.~Shparlinski]{Igor E. Shparlinski}
\address{I.S.: Department of Pure  Mathematics, University of New South Wales, Sydney, NSW 2052, Australia}
\email{igor.shparlinski@unsw.edu.au}

\author[D. Sutantyo]{Daniel Sutantyo}
\address{D.S.: Department of Computing, Macquarie University, Sydney,\linebreak NSW 2109, Australia}
\email{daniel.sutantyo@gmail.com}

\begin{abstract}
We study functional graphs generated by quadratic polynomials over prime fields.
We introduce efficient algorithms for methodical computations and provide the values of various direct and cumulative statistical parameters of interest. These include: the number of connected functional graphs,
  the number of graphs having a maximal cycle,
  the number of cycles of fixed size,
  the number of components of fixed size,
  as well as the shape of trees extracted from functional graphs. We particularly focus on connected functional  graphs, that is, the graphs which
contain only one component (and thus only one cycle).
Based on the results of our computations, we formulate several conjectures highlighting the similarities and differences between these functional graphs and random mappings.
\end{abstract}



\maketitle

\section{Introduction}

Let $\F_q$ be the finite field of $q$ elements
and
of characteristic $p$, with $p\ge 3$.
For a function $f:\F_q \to \F_q$, we define
the functional graph of $f$ as a directed graph $\cG_f$ on $q$ nodes
labelled by the elements of $\F_q$ where
there is an edge from $u$ to $v$ if and only if $f(u) = v$.
For any integer $n\ge 1$, let $f^{(n)}$ be the $n$-th iteration of $f$.

These graphs are particular as one can immediately observe that each connected component
of the graph $\cG_f$  has a unique cycle
(we treat fixed points as cycles of length $1$).
An example for the functional graph of $x^2+12 \pmod{31}$ is given in Figure~\ref{pic:connected_graph}.

                            				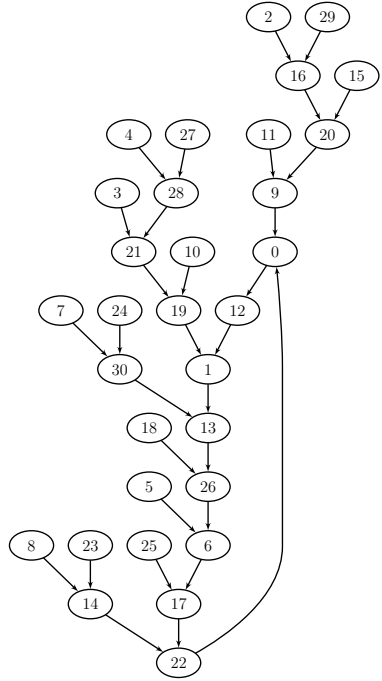
\begin{figure}
                            				\resizebox{5cm}{9cm}{%
                            					\begin{tikzpicture}[>=latex',line join=bevel,scale=0.60]
                            						\pgfsetlinewidth{1bp}
                            					\pgfsetcolor{black}
                            						\draw [->] (416.65bp,720.76bp) .. controls (412.29bp,712.28bp) and (406.85bp,701.71bp)  .. (397.3bp,683.15bp);
                            						\draw [->] (207.0bp,71.697bp) .. controls (207.0bp,63.983bp) and (207.0bp,54.712bp)  .. (207.0bp,36.104bp);
                            						\draw [->] (117.81bp,76.807bp) .. controls (135.0bp,65.665bp) and (160.62bp,49.062bp)  .. (188.4bp,31.053bp);
                            						\draw [->] (41.57bp,146.83bp) .. controls (51.75bp,136.94bp) and (65.524bp,123.55bp)  .. (84.204bp,105.38bp);
                            						\draw [->] (153.81bp,364.81bp) .. controls (171.0bp,353.67bp) and (196.62bp,337.06bp)  .. (224.4bp,319.05bp);
                            						\draw [->] (243.0bp,359.7bp) .. controls (243.0bp,351.98bp) and (243.0bp,342.71bp)  .. (243.0bp,324.1bp);
                            						\draw [->] (192.46bp,577.46bp) .. controls (185.73bp,568.4bp) and (177.1bp,556.79bp)  .. (163.51bp,538.49bp);
                            						\draw [->] (314.56bp,505.12bp) .. controls (308.8bp,496.34bp) and (301.52bp,485.26bp)  .. (289.4bp,466.82bp);
                            						\draw [->] (179.35bp,144.76bp) .. controls (183.71bp,136.28bp) and (189.15bp,125.71bp)  .. (198.7bp,107.15bp);
                            						\draw [->] (77.57bp,434.83bp) .. controls (87.75bp,424.94bp) and (101.52bp,411.55bp)  .. (120.2bp,393.38bp);
                            						\draw [->] (163.93bp,505.81bp) .. controls (171.21bp,496.55bp) and (180.66bp,484.52bp)  .. (195.09bp,466.16bp);
                            						\draw [->] (361.35bp,720.76bp) .. controls (365.71bp,712.28bp) and (371.15bp,701.71bp)  .. (380.7bp,683.15bp);
                            						\draw [->] (214.61bp,648.05bp) .. controls (213.07bp,640.35bp) and (211.21bp,631.03bp)  .. (207.46bp,612.28bp);
                            						\draw [->] (185.57bp,290.83bp) .. controls (195.75bp,280.94bp) and (209.52bp,267.55bp)  .. (228.2bp,249.38bp);
                            						\draw [->] (318.98bp,647.7bp) .. controls (319.86bp,639.98bp) and (320.92bp,630.71bp)  .. (323.05bp,612.1bp);
                            						\draw [->] (135.0bp,431.7bp) .. controls (135.0bp,423.98bp) and (135.0bp,414.71bp)  .. (135.0bp,396.1bp);
                            						\draw [->] (219.88bp,504.05bp) .. controls (217.99bp,496.26bp) and (215.7bp,486.82bp)  .. (211.2bp,468.28bp);
                            						\draw [->] (243.0bp,287.7bp) .. controls (243.0bp,279.98bp) and (243.0bp,270.71bp)  .. (243.0bp,252.1bp);
                            						\draw [->] (270.65bp,432.76bp) .. controls (266.29bp,424.28bp) and (260.85bp,413.71bp)  .. (251.3bp,395.15bp);
                            						\draw [->] (158.59bp,649.81bp) .. controls (166.26bp,640.55bp) and (176.23bp,628.52bp)  .. (191.44bp,610.16bp);
                            						\draw [->] (325.0bp,575.7bp) .. controls (325.0bp,567.98bp) and (325.0bp,558.71bp)  .. (325.0bp,540.1bp);
                            						\draw [->] (99.0bp,143.7bp) .. controls (99.0bp,135.98bp) and (99.0bp,126.71bp)  .. (99.0bp,108.1bp);
                            						\draw [->] (227.69bp,29.757bp) .. controls (263.67bp,50.127bp) and (334.0bp,97.994bp)  .. (334.0bp,161.0bp) .. controls (334.0bp,379.0bp) and (334.0bp,379.0bp)  .. (334.0bp,379.0bp) .. controls (334.0bp,419.08bp) and (330.38bp,465.41bp)  .. (326.78bp,503.97bp);
                            						\draw [->] (185.57bp,218.83bp) .. controls (195.75bp,208.94bp) and (209.52bp,195.55bp)  .. (228.2bp,177.38bp);
                            						\draw [->] (325.35bp,792.76bp) .. controls (329.71bp,784.28bp) and (335.15bp,773.71bp)  .. (344.7bp,755.15bp);
                            						\draw [->] (136.84bp,576.05bp) .. controls (139.1bp,568.14bp) and (141.85bp,558.54bp)  .. (147.2bp,539.79bp);
                            						\draw [->] (243.0bp,215.7bp) .. controls (243.0bp,207.98bp) and (243.0bp,198.71bp)  .. (243.0bp,180.1bp);
                            						\draw [->] (380.65bp,792.76bp) .. controls (376.29bp,784.28bp) and (370.85bp,773.71bp)  .. (361.3bp,755.15bp);
                            						\draw [->] (375.43bp,650.15bp) .. controls (366.69bp,640.6bp) and (355.17bp,627.99bp)  .. (338.55bp,609.82bp);
                            						\draw [->] (234.65bp,144.76bp) .. controls (230.29bp,136.28bp) and (224.85bp,125.71bp)  .. (215.3bp,107.15bp);
                            						\draw [->] (215.35bp,432.76bp) .. controls (219.71bp,424.28bp) and (225.15bp,413.71bp)  .. (234.7bp,395.15bp);
                            					\begin{scope}
                            						\definecolor{strokecol}{rgb}{0.0,0.0,0.0};
                            						\pgfsetstrokecolor{strokecol}
                            						\draw (135.0bp,450.0bp) ellipse (27.0bp and 18.0bp);
                            						\draw (135.0bp,450.0bp) node {24};
                            					\end{scope}
                            					\begin{scope}
                            						\definecolor{strokecol}{rgb}{0.0,0.0,0.0};
                            						\pgfsetstrokecolor{strokecol}
                            						\draw (171.0bp,162.0bp) ellipse (27.0bp and 18.0bp);
                            						\draw (171.0bp,162.0bp) node {25};
                            					\end{scope}
                            					\begin{scope}
                            						\definecolor{strokecol}{rgb}{0.0,0.0,0.0};
                            						\pgfsetstrokecolor{strokecol}
                            						\draw (243.0bp,234.0bp) ellipse (27.0bp and 18.0bp);
                            						\draw (243.0bp,234.0bp) node {26};
                            					\end{scope}
                            					\begin{scope}
                            						\definecolor{strokecol}{rgb}{0.0,0.0,0.0};
                            						\pgfsetstrokecolor{strokecol}
                            						\draw (218.0bp,666.0bp) ellipse (27.0bp and 18.0bp);
                            						\draw (218.0bp,666.0bp) node {27};
                            					\end{scope}
                            					\begin{scope}
                            						\definecolor{strokecol}{rgb}{0.0,0.0,0.0};
                            						\pgfsetstrokecolor{strokecol}
                            						\draw (389.0bp,666.0bp) ellipse (27.0bp and 18.0bp);
                            						\draw (389.0bp,666.0bp) node {20};
                            					\end{scope}
                            					\begin{scope}
                            						\definecolor{strokecol}{rgb}{0.0,0.0,0.0};
                            						\pgfsetstrokecolor{strokecol}
                            						\draw (152.0bp,522.0bp) ellipse (27.0bp and 18.0bp);
                            						\draw (152.0bp,522.0bp) node {21};
                            					\end{scope}
                            					\begin{scope}
                            						\definecolor{strokecol}{rgb}{0.0,0.0,0.0};
                            						\pgfsetstrokecolor{strokecol}
                            						\draw (207.0bp,18.0bp) ellipse (27.0bp and 18.0bp);
                            						\draw (207.0bp,18.0bp) node {22};
                            					\end{scope}
                            					\begin{scope}
                            						\definecolor{strokecol}{rgb}{0.0,0.0,0.0};
                            						\pgfsetstrokecolor{strokecol}
                            						\draw (99.0bp,162.0bp) ellipse (27.0bp and 18.0bp);
                            						\draw (99.0bp,162.0bp) node {23};
                            					\end{scope}
                            					\begin{scope}
                            						\definecolor{strokecol}{rgb}{0.0,0.0,0.0};
                            						\pgfsetstrokecolor{strokecol}
                            						\draw (204.0bp,594.0bp) ellipse (27.0bp and 18.0bp);
                            						\draw (204.0bp,594.0bp) node {28};
                            					\end{scope}
                            					\begin{scope}
                            						\definecolor{strokecol}{rgb}{0.0,0.0,0.0};
                            						\pgfsetstrokecolor{strokecol}
                            						\draw (389.0bp,810.0bp) ellipse (27.0bp and 18.0bp);
                            						\draw (389.0bp,810.0bp) node {29};
                            					\end{scope}
                            					\begin{scope}
                            						\definecolor{strokecol}{rgb}{0.0,0.0,0.0};
                            						\pgfsetstrokecolor{strokecol}
                            						\draw (243.0bp,378.0bp) ellipse (27.0bp and 18.0bp);
                            						\draw (243.0bp,378.0bp) node {1};
                            					\end{scope}
                            					\begin{scope}
                            						\definecolor{strokecol}{rgb}{0.0,0.0,0.0};
                            						\pgfsetstrokecolor{strokecol}
                            						\draw (325.0bp,522.0bp) ellipse (27.0bp and 18.0bp);
                            						\draw (325.0bp,522.0bp) node {0};
                            					\end{scope}
                            					\begin{scope}
                            						\definecolor{strokecol}{rgb}{0.0,0.0,0.0};
                            						\pgfsetstrokecolor{strokecol}
                            						\draw (132.0bp,594.0bp) ellipse (27.0bp and 18.0bp);
                            						\draw (132.0bp,594.0bp) node {3};
                            					\end{scope}
                            					\begin{scope}
                            						\definecolor{strokecol}{rgb}{0.0,0.0,0.0};
                            						\pgfsetstrokecolor{strokecol}
                            						\draw (317.0bp,810.0bp) ellipse (27.0bp and 18.0bp);
                            						\draw (317.0bp,810.0bp) node {2};
                            					\end{scope}
                            					\begin{scope}
                            						\definecolor{strokecol}{rgb}{0.0,0.0,0.0};
                            						\pgfsetstrokecolor{strokecol}
                            						\draw (171.0bp,234.0bp) ellipse (27.0bp and 18.0bp);
                            						\draw (171.0bp,234.0bp) node {5};
                            					\end{scope}
                            					\begin{scope}
                            						\definecolor{strokecol}{rgb}{0.0,0.0,0.0};
                            						\pgfsetstrokecolor{strokecol}
                            						\draw (146.0bp,666.0bp) ellipse (27.0bp and 18.0bp);
                            						\draw (146.0bp,666.0bp) node {4};
                            					\end{scope}
                            					\begin{scope}
                            						\definecolor{strokecol}{rgb}{0.0,0.0,0.0};
                            						\pgfsetstrokecolor{strokecol}
                            						\draw (63.0bp,450.0bp) ellipse (27.0bp and 18.0bp);
                            						\draw (63.0bp,450.0bp) node {7};
                            					\end{scope}
                            					\begin{scope}
                            						\definecolor{strokecol}{rgb}{0.0,0.0,0.0};
                            						\pgfsetstrokecolor{strokecol}
                            						\draw (243.0bp,162.0bp) ellipse (27.0bp and 18.0bp);
                            						\draw (243.0bp,162.0bp) node {6};
                            					\end{scope}
                            					\begin{scope}
                            						\definecolor{strokecol}{rgb}{0.0,0.0,0.0};
                            						\pgfsetstrokecolor{strokecol}
                            						\draw (325.0bp,594.0bp) ellipse (27.0bp and 18.0bp);
                            						\draw (325.0bp,594.0bp) node {9};
                            					\end{scope}
                            					\begin{scope}
                            						\definecolor{strokecol}{rgb}{0.0,0.0,0.0};
                            						\pgfsetstrokecolor{strokecol}
                            						\draw (27.0bp,162.0bp) ellipse (27.0bp and 18.0bp);
                            						\draw (27.0bp,162.0bp) node {8};
                            					\end{scope}
                            					\begin{scope}
                            						\definecolor{strokecol}{rgb}{0.0,0.0,0.0};
                            						\pgfsetstrokecolor{strokecol}
                            						\draw (317.0bp,666.0bp) ellipse (27.0bp and 18.0bp);
                            						\draw (317.0bp,666.0bp) node {11};
                            					\end{scope}
                            					\begin{scope}
                            						\definecolor{strokecol}{rgb}{0.0,0.0,0.0};
                            						\pgfsetstrokecolor{strokecol}
                            						\draw (224.0bp,522.0bp) ellipse (27.0bp and 18.0bp);
                            						\draw (224.0bp,522.0bp) node {10};
                            					\end{scope}
                            					\begin{scope}
                            						\definecolor{strokecol}{rgb}{0.0,0.0,0.0};
                            						\pgfsetstrokecolor{strokecol}
                            						\draw (243.0bp,306.0bp) ellipse (27.0bp and 18.0bp);
                            						\draw (243.0bp,306.0bp) node {13};
                            					\end{scope}
                            					\begin{scope}
                            						\definecolor{strokecol}{rgb}{0.0,0.0,0.0};
                            						\pgfsetstrokecolor{strokecol}
                            						\draw (279.0bp,450.0bp) ellipse (27.0bp and 18.0bp);
                            						\draw (279.0bp,450.0bp) node {12};
                            					\end{scope}
                            					\begin{scope}
                            						\definecolor{strokecol}{rgb}{0.0,0.0,0.0};
                            						\pgfsetstrokecolor{strokecol}
                            						\draw (425.0bp,738.0bp) ellipse (27.0bp and 18.0bp);
                            						\draw (425.0bp,738.0bp) node {15};
                            					\end{scope}
                            					\begin{scope}
                            						\definecolor{strokecol}{rgb}{0.0,0.0,0.0};
                            						\pgfsetstrokecolor{strokecol}
                            						\draw (99.0bp,90.0bp) ellipse (27.0bp and 18.0bp);
                            						\draw (99.0bp,90.0bp) node {14};
                            					\end{scope}
                            					\begin{scope}
                            						\definecolor{strokecol}{rgb}{0.0,0.0,0.0};
                            						\pgfsetstrokecolor{strokecol}
                            						\draw (207.0bp,90.0bp) ellipse (27.0bp and 18.0bp);
                            						\draw (207.0bp,90.0bp) node {17};
                            					\end{scope}
                            					\begin{scope}
                            						\definecolor{strokecol}{rgb}{0.0,0.0,0.0};
                            						\pgfsetstrokecolor{strokecol}
                            						\draw (353.0bp,738.0bp) ellipse (27.0bp and 18.0bp);
                            						\draw (353.0bp,738.0bp) node {16};
                            					\end{scope}
                            					\begin{scope}
                            						\definecolor{strokecol}{rgb}{0.0,0.0,0.0};
                            						\pgfsetstrokecolor{strokecol}
                            						\draw (207.0bp,450.0bp) ellipse (27.0bp and 18.0bp);
                            						\draw (207.0bp,450.0bp) node {19};
                            					\end{scope}
                            					\begin{scope}
                            						\definecolor{strokecol}{rgb}{0.0,0.0,0.0};
                            						\pgfsetstrokecolor{strokecol}
                            						\draw (171.0bp,306.0bp) ellipse (27.0bp and 18.0bp);
                            						\draw (171.0bp,306.0bp) node {18};
                            					\end{scope}
                            					\begin{scope}
                            						\definecolor{strokecol}{rgb}{0.0,0.0,0.0};
                            						\pgfsetstrokecolor{strokecol}
                            						\draw (135.0bp,378.0bp) ellipse (27.0bp and 18.0bp);
                            						\draw (135.0bp,378.0bp) node {30};
                            					\end{scope}
                            					\end{tikzpicture}
                            				}
                            				\caption{The functional graph of $X^2+12 \pmod{31}$}
				                         \label{pic:connected_graph}
                            				\end{figure}

Recently, there have been an increasing interest in studying, theoretically and experimentally,
the graphs $\cG_f$ generated
by polynomials $f \in \F_q[X]$ of small degree (such as quadratic polynomials), and how they differ, or not, from random mappings~\cite{FO2}.
We refer to~\cite{BGTW,BrGa,BuSch,FlGar,KLMMSS, OstSha} and the references therein.

In this paper, we concentrate on the case of quadratic polynomials over prime fields.
In fact, up to isomorphism we only need to consider polynomials $f_a(X) = X^2 +a$, $a \in \F_p$ (see the proof of~\cite[Theorem~2.1]{KLMMSS}).
For simplicity, we use $\cG_a = \cG_{f_a}$ to denote the functional graph generated by $f_a$.
For this case, in~\cite[Section 4]{KLMMSS} the authors have provided numerical data for the number of distinct graphs $\cG_a$,
the statistics of cyclic points, the number of connected components, as well as the most popular component size.

Different from the aspects in~\cite{KLMMSS}, we consider several questions related to
distributions of cyclic points and sizes of connected components of $\cG_a$ when $a$ runs through the elements in $\F_p$.
In particular, we are interested in characterising connected functional  graphs $\cG_a$, that is, the graphs which
contain only one component (and thus only one cycle).

In this paper, we focus on characterising the functional graphs by providing direct parameters such as the number of (connected) components.
We then characterise various cumulative parameters, such as the number of cyclic points and the shape of trees extracted from functional graphs.
We highlight similarities and differences  between functional graphs~\cite{KLMMSS} and random mappings~\cite{FO2},
and we also pay much attention to features of connected functional graphs.
While obtaining theoretic results for these questions remains a challenge, we introduce efficient algorithms and
present new interesting results of numerical experiments.

The rest of the paper is structured as follows.
In Section~\ref{sect:counting}, we develop a fast algorithm that determines whether a functional graph is connected,
which is used to compute the number of connected functional graphs.
In Section~\ref{sect:cyclic},
we compare the number of cyclic points in connected graphs with those in all graphs modulo $p$.
In Section~\ref{sect:smallcyclic} and Section~\ref{sect:smallsize} respectively, we consider the number
of components with small number  of cyclic points and with small size.
Finally, in Section~\ref{sect:tree} we illustrate the statistics of trees in functional graphs.

Throughout the paper, we use the Landau symbol $O$. Recall that the assertion $U=O(V)$ is equivalent to the inequality $|U|\le cV$ with some absolute constant $c>0$.
  To emphasise the dependence of the implied
constant $c$ on some parameter (or a list of parameters) $\rho$, we write $U=O_{\rho}(V)$.
We also use the asymptotic symbol $\sim$.

\section{Counting connected graphs}
\label{sect:counting}

In this section, we introduce a new efficient algorithm that quickly  detects  connected functional graphs, and formulate some conjectures for the number of connected graphs based on our computations.

\subsection{Preliminaries and informal ideas of the algorithm}

Let $\cI_p$ be the set $a \in \F_p$ such that $\cG_a$ is connected.
We also denote by $I_p = \# \cI_p$ the
number of connected graphs $\cG_a$ with $a \in \F_p$.
Clearly the graph $\cG_0$ is not connected, and also
by~\cite[Corollary~18~(a)]{VaSha} $\cG_{-2}$ is also not connected if $p>3$,
and so $\cI_p \subseteq \F_p \setminus \{0,-2\}$ if $p>3$.
In fact, the functional graphs with values $a = 0$ and $a = -2$ lead to graphs with a particular
group structure (and thus the structure of these graphs deviates significantly
from the other graphs, see~\cite{VaSha}).

Essentially in~\cite[Algorithm 3.1]{KLMMSS}, a rigorous deterministic algorithm using Floyd's cycle detection algorithm and needing $O(p)$ function evaluations (that is, of complexity $p^{1+o(1)}$)
has been used to test whether $\cG_a$ is a connected graph.
Instead of evaluating $I_p$ via this algorithm which would need $O(p^2)$ function evaluations,
we introduce a more efficient heuristic approach in practice, which is specifically useful for computations of a family of graphs (not just a single graph).

The main idea is to first check quickly whether $\cG_a$ has more than one small cycle (i.e., more than one component).
A graph $\cG_a$ has a component  with a \textit{cycle} of size   $\ell$ if and only if the
equation $f_a^{(\ell)}(u) = u$ has a solution $u$ which is not a solution to any of the equations
$f_a^{(k)}(u) = u$ with $1 \le k < \ell$.
The roots of $f_a^{(\ell)}(u) = u$ are the \textit{cyclic points} in the graph.
For this we need the \textit{dynatomic polynomials}
$$
F_a^{(\ell)}(X)  = \prod_{r \mid\ell} \(f_a^{(r)}(X) - X\)^{\mu(\ell/r)},
$$
where $\mu(k)$ is the M{\"o}bius function, see~\cite[Section~4.1]{Silv}.
Moreover, we have
$$
f_a^{(n)}(X) - X = \prod_{\ell \mid n} F_a^{(\ell)}(X), \quad n=1,2,\ldots.
$$

For example
$$
F_a^{(1)}(X)  = X^2 -X +a \mand F_a^{(2)}(X)  = X^2 + X + a +1
$$
and
$$
 F_a^{(3)}(X)  = \(f_a^{(3)}(X) - X\)/ \(f_a^{(1)}(X) - X\).
$$

Clearly, if  $\cG_a$ has a cycle of length $\ell$, then any point in this cycle is a root of the polynomial $F_a^{(\ell)}(X)$.
However, the roots of $F_a^{(\ell)}(X)$ might be not all lying in cycles of length $\ell$; for instance see~\cite[Example~4.2]{Silv}.
Certainly, $\cG_a$ is not connected if $F_a^{(\ell)}(X)$ has a root for two distinct values
of $\ell =\ell_1,\ell_2$ with $\ell_1\nmid \ell_2$ and $\ell_2 \nmid \ell_1$. Alternatively,  if  $F_a^{(\ell)}(X)$ has more than $\ell$ distinct roots, this indicates that
$\cG_a$ has at least two cycles, which again implies that $\cG_a$ has more than one
connected component.

As we  show later, it turns out that this occurs frequently and thus we can rule out the connectivity of  most of the
graph $\cG_a$, $a \in \F_p$ quickly. A relatively small number of remaining suspects can be checked via
the rigorous deterministic algorithm from~\cite[Algorithm 3.1]{KLMMSS}.

\subsection{Algorithm}

Algorithm~\ref{algo:one_component} is to determine whether a graph is connected or not,
where we in fact use $f_a^{(\ell)}(X)$ instead of $F_a^{(\ell)}(X)$.

\begin{algorithm}
\begin{algorithmic}[1]
	\REQUIRE prime $p$, integer $a \pmod p$ and integer $L$.
	\ENSURE returns true if $X^2+a \pmod p$ generates a connected functional graph, and false otherwise
	\STATE $cycles \leftarrow 0$
	\STATE $g_1 \leftarrow \gcd(X^p-X,f_a^{(1)}(X)-X)$
	\IF {$\deg g_1 \ge 1$}
		\IF{$\deg g_1 = 2$}
			\RETURN false
		\ENDIF
	    \STATE $cycles \leftarrow cycles + 1$
	\ENDIF
	\FOR{$i \leftarrow 2$ to $L$}
		\STATE $g_i \leftarrow \gcd(X^p-X,f_a^{(i)}(X)-X)$
		\IF{$\deg g_i > i$}
			\RETURN false
		\ELSIF {$\deg g_i = i$}
			\STATE $cycles \leftarrow cycles + 1$
		\ENDIF
		\IF{$cycles > 1$}
			\RETURN false
		\ENDIF
	\ENDFOR
  \FOR{$j \leftarrow 0$ to $p-1$}
    \STATE start traversal from node $j$
    \IF{two cycles are detected}
      \RETURN false
    \ENDIF
  \ENDFOR
	\RETURN true
\end{algorithmic}
\caption{Determine if $\cG_a$ is a connected graph}
\label{algo:one_component}
\end{algorithm}

The algorithm starts by checking if there is any cycle of size 1 in the graph.
Since $X^p-X$ only contains simple roots and $f_a^{(1)}(X)$ has degree 2,
if $\gcd(X^p-X,f_a^{(1)}(X)) > 1$, then there are two cycles of size 1 and
thus two separate components in the graph.
Otherwise, there is at most one component with a cycle of size 1 in the graph $\cG_a$.

Next, we compute $g_i = \gcd(X^p-X,f_a^{(i)}(X)-X)$ from $i=2$ until $L$ while keeping track of
the number of cycles that has been detected.
Here, we have several possibilities:
\begin{itemize}
\item if $\deg g_i < i$, then there are no cycle of size $i$ in the graph.
\item if $\deg g_i = i$, then there is exactly one cycle of size $i$.
\item if $\deg g_i > i$, then there are at least two different cycles in the
graph.
\end{itemize}

When $\deg g_i < i$, there are no cycle of size $i$ since there are
not enough roots to form one.
Similarly, if $\deg g_i > i$, then there are more than $i$ cyclic points
in the graph, of which at least $i$ of them form one cycle, and
so there are more than one cycle in the graph.

Finally, if at this stage the algorithm detects $\deg g_i = i$, then there is exactly one cycle of size $i$.
By contradiction, if there is no cycle of size $i$, then there must be at least two cycles of size less than $i$,
and so  we would have detected that $cycles > 1$ at a previous iteration,
thus returning `false'.

Once we are done with the first loop, either we have found one cycle with size at most  $L$,
or we have not found any small cycles at all.  We then proceed with a graph traversal
until we find two cycles.

\subsection{Statistics of the number of connected graphs}

We implement  Algorithm~\ref{algo:one_component} by using NTL~\cite{NTL} and PARI/GP~\cite{Pari}, choosing $L=5$ in our computations.
 We collect values of $I_p$ for some primes (as shown in Table~\ref{table:one_component_count})
that lead us to the following conjecture:

\begin{conj}     \label{conj:Ip}
$I_p \sim \sqrt{2p}$ as $p\to \infty$.
\end{conj}

Here, we also pose a weaker conjecture:

\begin{conj} \label{conj:Ip 1}
For any prime $p$, $I_p \ge 1$.
\end{conj}

Conjecture~\ref{conj:Ip 1} predicts that there always exists a connected functional graph generated by quadratic polynomials modulo $p$.
Indeed, according to our computations, Conjecture~\ref{conj:Ip 1} is true for all primes $p \le 100 000$.

\begin{table}[H]
{\small
	\begin{tabular}{r r r}
	\hline \\[-2.2ex]
	\multicolumn{1}{c}{$p$} &
	\multicolumn{1}{c}{$I_p$} &
	\multicolumn{1}{c}{$\sqrt{2p}$}\\[0.2ex]
	\hline\\[-1.5ex]
	  500,009 & 1,038 & 1,000.009  \\
	  500,029 & 1,002 & 1,000.029  \\
	  500,041 &  956 & 1,000.041  \\
	  500,057 & 1,026 & 1,000.057  \\
	  500,069 &  995 & 1,000.069  \\
	  500,083 &  987 & 1,000.083  \\
	  500,107 &  994 & 1,000.107  \\
    500,111 & 1,010 & 1,000.111  \\
	  500,113 & 1,019 & 1,000.113  \\
	  500,119 &  920 & 1,000.119  \\
	  500,153 & 1,033 & 1,000.153  \\
	  500,167 & 1,005 & 1,000.167  \\
	1,000,003 & 1,369 & 1,414.296  \\
	2,000,003 & 1,909 & 2,000.001  \\
	3,000,017 & 2,478 & 2,449.497  \\
	4,000,037 & 2,838 & 2,828.440  \\
	\hline\\
	\end{tabular}
}
\caption{The number of connected graphs modulo $p$}
\label{table:one_component_count}
\end{table}

We also investigate the existence of connected functional graphs having (only) one cycle of size $1$.

If the graph $\cG_a$ is connected and has one cycle of size $1$,
then the equation $X^2+a=X$ has two identical roots (corresponding to fixed points), and so $a=1/4$ and the root $x=1/2$.
Thus, we only need to check the graph generated by $X^2+1/4$ in $\F_p$.

We have tested all the primes up to 100000 and we only have found two such examples: one is $X^2+1$ in $\F_3$, and
the other is $X^2+2$ in $\F_7$.
Furthermore, we have:

\begin{prop}
For any prime $p$ with $p\equiv \textrm{$5$ or $11$} \pmod{12}$, there is no functional graph $\cG_a$ having only one cycle of size $1$.
\end{prop}

\begin{proof}
Note that we only need to consider the graph $\cG_{1/4}$.
Since $1/2$ is a fixed point of $\cG_{1/4}$ and there is an edge from $-1/2$ to $1/2$,
 we consider the equation $X^2+1/4 = -1/2$ in $\F_p$, that is, whether $-3$ is a square in $\F_p$.
However, if $p\equiv \textrm{$5$ or $11$} \pmod{12}$, $-3$ is not a square in $\F_p$.
Then, the in-degree of $-1/2$ is zero, and so $\cG_{1/4}$ must have more than one cycle.
This completes the proof.
\end{proof}

So, we pose the following conjecture:

\begin{conj}
For any prime $p>7$, there is no functional graph $\cG_a$ having only one cycle of size $1$.
\end{conj}

\section{Counting cyclic points in functional graphs}
\label{sect:cyclic}

We now assess the number of cyclic points in functional graphs modulo $p$.
For the minimal and maximal numbers of cyclic points in graphs $\cG_a$, we refer to~\cite[Table~4.1]{KLMMSS},
where the cases $a=0,-2$ are excluded.
Roughly speaking, the reason why these two cases are excluded is that the number of cyclic points is maximized on the cases
$a=0,-2$ quite often; see \cite[Section 4.3]{KLMMSS} for more details.
In this section, we also follow this convention.

Let $C_a$ be the total number of cyclic points of $\cG_a$, and let $c_a$
be the largest number of cyclic points in a single component of $\cG_a$.
Clearly we have $C_a \ge c_a$ for any $a \in \F_p$  and $C_a = c_a$ when $a \in \cI_p$.

Furthermore, we define the average and largest values of these quantities:
\begin{align*}
&\overline{C_p} = \frac{1}{p-2} \sum_{a \in \F_p \setminus \{0,-2\}} C_a,     &\bC_p= \max \left\{C_a :~a \in \F_p \setminus \{0,-2\} \right\}; \\
&\overline{c_p} = \frac{1}{p-2} \sum_{a \in \F_p \setminus \{0,-2\}} c_a,       &\bc_p = \max \left\{ c_a:~a \in \F_p  \setminus \{0,-2\}\right\};\\
&\overline{c_p}^*=    \frac{1}{I_p} \sum_{a \in \cI_p} c_a,       &\bc_p^*= \max \left\{c_a:~a \in \cI_p \right\}.
\end{align*}
We remark again that $\cI_p \subseteq \F_p \setminus \{0,-2\}$ if $p>3$.


Numerical experiments in~\cite[Section~4.3]{KLMMSS}  suggest that the average number of cyclic points modulo $p$, taken
over all graphs  modulo $p$ (excluding $a=0,-2$), is $\sqrt{\pi p/2}$, which is consistent with the behaviour of random maps
(see~\cite[Theorem~2(ii)]{FO2}).
Here we show that this is not the case for connected graphs (see Table~\ref{table:cyclic_points}).
In that case,  $\overline{c_p}^*$ is smaller than $\overline{C_p}$, i.e.
there are  fewer cyclic points than those for non-connected graphs on average.
Notice that both $\overline{c_p}^*$ and $\overline{c_p}$ are both close to
$\sqrt{2p/\pi}$ (and although close to each other, $\overline{c_p}^*$ is slightly larger).


\begin{table}[H]
{\small
	\begin{tabular}{r r r r r r r r}
	\hline \\[-2.2ex]
	\multicolumn{1}{c}{$p$} &
        \multicolumn{1}{c}{$\overline {C_p} $} &
	\multicolumn{1}{c}{$\sqrt{\pi p/2}$} &
  \multicolumn{1}{c}{$\overline {c_p} $} &
  \multicolumn{1}{c}{$\overline {c_p}^* $} &
	\multicolumn{1}{c}{$\sqrt{2p/\pi}$}\\[0.2ex]
	\hline\\[-1.5ex]
	 500,009 &  886.224 &  886.235 &  553.445 &  573.355 &  564.194\\
   500,029 &  885.990 &  886.253 &  553.312 &  587.750 &  564.205\\
   500,041 &  885.069 &  886.263 &  553.175 &  568.208 &  564.212\\
   500,057 &  884.963 &  886.277 &  552.870 &  586.037 &  564.221\\
   500,069 &  885.831 &  886.288 &  552.952 &  558.285 &  564.229\\
   500,083 &  884.970 &  886.300 &  552.692 &  564.995 &  564.236\\
   500,107 &  884.507 &  886.322 &  552.674 &  562.690 &  564.250\\
   500,111 &  884.341 &  886.325 &  552.157 &  575.976 &  564.252\\
   500,113 &  885.160 &  886.327 &  552.988 &  568.057 &  564.253\\
   500,119 &  884.559 &  886.332 &  552.597 &  569.750 &  564.257\\
   500,153 &  884.834 &  886.363 &  552.900 &  589.146 &  564.276\\
   500,167 &  885.756 &  886.375 &  552.525 &  560.095 &  564.284\\
   600,011 &  969.139 &  970.822 &  605.632 &  611.914 &  618.044\\
   700,001 & 1,047.771 & 1,048.599 &  654.317 &  667.624 &  667.559\\
   800,011 & 1,120.427 & 1,121.006 &  700.047 &  703.061 &  713.655\\
   900,001 & 1,188.822 & 1,188.999 &  742.619 &  762.673 &  756.940\\
  1,000,003 & 1,252.452 & 1,253.316 &  782.026 &  793.388 &  797.886\\
  2,000,003 & 1,772.078 & 1,772.455 & 1,106.815 & 1,134.598 & 1,128.380\\
	\hline\\
	\end{tabular}
}
\caption{Average number of cyclic points in graphs modulo $p$ (excluding $a=0,-2$)}
\label{table:cyclic_points}
\end{table}

In Table~\ref{table:max_cyclic_points}, one can see that the largest cycles usually
do not appear in the connected graphs, which appears surprising and shows the existence of components with
a large cycle even when the graph is disconnected.
In addition, the difference $\bc_p-\bc_p^*$ is large, while the difference of $\bC_p$ and $\bc_p$ is small.


\begin{table}[H]
{\small
	\begin{tabular}{r r r r r r}
	\hline \\[-2.2ex]
	\multicolumn{1}{c}{$p$} &
	\multicolumn{1}{c}{$\bC_p$} & 
	\multicolumn{1}{c}{$\bc_p$} & 
	\multicolumn{1}{c}{$\bc_p^*$} \\[0.2ex] 
	\hline\\[-1.5ex]
	 500,009 & 3,578 & 3,164 & 2,319 \\
   500,029 & 3,620 & 3,291 & 2,327 \\
   500,041 & 3,798 & 3,118 & 2,333 \\
   500,057 & 3,468 & 3,319 & 2,423 \\
   500,069 & 3,556 & 3,129 & 2,089 \\
   500,083 & 3,596 & 3,050 & 2,131 \\
   500,107 & 3,527 & 3,232 & 2,643 \\
   500,111 & 3,732 & 3,237 & 2,244 \\
   500,113 & 3,805 & 3,232 & 2,335 \\
   500,119 & 3,873 & 3,142 & 2,275 \\
   500,153 & 3,472 & 3,380 & 2,754 \\
   500,167 & 3,644 & 3,159 & 2,770 \\
   600,011 & 3,847 & 3,488 & 3,265 \\
   700,001 & 4,350 & 3,670 & 2,950 \\
   800,011 & 4,600 & 4,242 & 3,208 \\
   900,001 & 4,997 & 4,274 & 3,245 \\
  1,000,003 & 5,101 & 4,639 & 3,117\\
  2,000,003 & 7,637 & 6,848 & 4,309\\
	\hline\\
	\end{tabular}
}
\caption{Maximum number of cyclic points in graphs modulo $p$ (excluding $a=0,-2$)}
\label{table:max_cyclic_points}
\end{table}

Let us also define the following three families of parameters $a$ on which the values
$\bC_p$, $\bc_p$ and $\bc_p^*$ are achieved, that is
\begin{align*}
&\sA_p = \left\{a\in \F_p\setminus \{0,-2\}:~C_a =\bC_p \right\}, \\
&\sB_p = \left\{a\in \F_p\setminus \{0,-2\}:~c_a =\bc_p \right\},\\
& \sB_p^* = \left\{a\in \cI_p:~c_a =\bc_p^* \right\}.
\end{align*}
It is certainly interesting to compare the sizes $A_p =\# \sA_p$, $B_p =\# \sB_p$
and  $B_p^* =\# \sB_p^*$   and also investigate the
mutual intersections between these families.


We find that typically these sets have one value of $a$ in common,
and  rarely more than two.
As $p$ increases, the frequency of the sets having $2$ or more elements decreases,
but does not disappear completely,
as can be seen in Table~\ref{table:size_of_ABsets}.

\begin{table}[H]
{\small
	\begin{tabular}{l | rrr | rrr | rrr}
	\hline \\[-2.2ex]
  &
  \multicolumn{3}{c}{$A_p$} &
  \multicolumn{3}{c}{$B_p$} &
  \multicolumn{3}{c}{$B_p^*$} \\
  range of $p$ &
  \multicolumn{1}{c}{$=1$} & \multicolumn{1}{c}{$=2$} & \multicolumn{1}{c}{$\ge3$} &
  \multicolumn{1}{c}{$=1$} & \multicolumn{1}{c}{$=2$} & \multicolumn{1}{c}{$\ge3$} &
  \multicolumn{1}{c}{$=1$} & \multicolumn{1}{c}{$=2$} & \multicolumn{1}{c}{$\ge3$}\\
	\hline\\[-1.5ex]
  $[3,10^4]$                 & 1,182 & 39 & 7 & 1,159 & 65 & 4 & 1,193 & 35 & 0\\
  $[10^4,2\cdot10^4]$        & 1,013 & 20 & 0 & 1,010 & 22 & 1& 1,019 & 14 & 0 \\
  $[2\cdot10^4,3\cdot10^4]$  &  967 & 14 & 2 &  970 & 13 & 0&  976 &  7 & 0 \\
  $[3\cdot10^4,4\cdot10^4]$  &  949 &  9 & 0 &  941 & 17 & 0&  950 &  8 & 0 \\
  $[4\cdot10^4,5\cdot10^4]$  &  921 &  8 & 1 &  921 &  9 & 0&  926 &  4 & 0 \\
  $[5\cdot10^4,6\cdot10^4]$  &  915 &  9 & 0 &  920 &  4 & 0 &  921 &  3 & 0\\
  $[6\cdot10^4,7\cdot10^4]$  &  868 & 10 & 0 &  872 &  6 & 0 &  868 &  9 & 1\\
  $[7\cdot10^4,8\cdot10^4]$  &  895 &  7 & 0 &  897 &  5 & 0 &  899 &  3 & 0\\
  $[8\cdot10^4,9\cdot10^4]$  &  869 &  7 & 0 &  869 &  7 & 0 &  866 & 10 & 0\\
  $[9\cdot10^4,10^5]$        &  874 &  5 & 0 &  878 &  1 & 0 &  876 &  3 & 0\\
  $[10^5,10^5+10^3]$         &   81 &  0 & 0 &   79 &  2 & 0 &   81 &  0 & 0\\
  $[10^6,10^6+10^3]$         &   74 &  1 & 0 &   75 &  0 & 0 &   74 &  1 & 0\\
	\hline\\
	\end{tabular}
}
\caption{Values of $A_p$, $B_p$, and $B_p^*$}
\label{table:size_of_ABsets}
\end{table}

For the set intersections, we start with
$\sA_p \cap \sB_p^*$. With Table~\ref{table:cyclic_points}, we have observed that
$\overline{C_p} > \overline{c_p}^*$, thus it is reasonable to expect that
$\sA_p \cap \sB_p^*$ is empty.
We remark that if $\sA_p \cap \sB_p^*$ is not empty,  then $\bC_p=\bc_p=\bc_p^*$, and so for any $a\in \sB_p$ the graph $\cG_a$ is connected, and thus $\sB_p = \sB_p^*$.
Therefore, for any prime $p$, if $\bc_p < \bC_p$, then we must have that $\sA_p \cap \sB_p^*$ is empty.
Our experiments with odd prime $p < 10^5$ counted only 20 occurrences of primes where
the intersection is non-empty and in fact contains only one value of $a$, shown in Table~\ref{table:ApBpstar}.

\begin{table}[H]
{\small
	\begin{tabular}{cccc}
	\hline \\[-2.2ex]
   $p$ & value of $a$ & $p$ & value of $a$ \\
	\hline\\[-1.5ex]
   3  & 2      &    271 &    147 \\
   5  & 1      &  2,647 &  1,445 \\
   7  & 3      &  3,613 &  2,653 \\
   11 & 6      &  6,131 &  3,555 \\
   13 & 1      &  6,719 &    107 \\
   17 & 3      & 17,921 &  8,370 \\
   19 & 13     & 18,077 & 15,557 \\
   29 & 4      & 36,229 &  2,229 \\
  157 & 141    & 53,611 & 23,630 \\
  191 & 97     & 64,667 & 60,638 \\
	\hline\\
	\end{tabular}
}
\caption{Values of $p$ with non-empty $\sA_p \cap \sB_p^*$}
\label{table:ApBpstar}
\end{table}

Since we have observed only one value of $a$ for each prime $p$ in the above table,
we conjecture that:

\begin{conj}
 For any prime $p \ge 3$, we have $\#\(\sA_p \cap \sB_p^*\) \le 1$.
\end{conj}

We also consider the intersection $\sB_p \cap \sB_p^*$; see Table~\ref{table:BpBpstar}.
Clearly, if $\sB_p \cap \sB_p^*$ is not empty, then we have $\bc_p=\bc_p^*$.
One could expect the number of primes with non-empty intersections
to decrease as $p$ increases, however even if our experiments show some reduction overall, it remains unclear.

\begin{table}[H]
{\small
	\begin{tabular}{lrrc}
	\hline \\[-2.2ex]
   range of $p$ &
  freq &
  \#primes &
  \% \\
	\hline\\[-1.5ex]
 $[3,10^4]$                 & 104 & 1,228 & 8.06\% \\
 $[10^4,2\cdot10^4]$        &  35 & 1,033 & 3.19\% \\
 $[2\cdot10^4,3\cdot10^4]$  &  32 &  983 & 3.26\% \\
 $[3\cdot10^4,4\cdot10^4]$  &  20 &  958 & 1.98\% \\
 $[4\cdot10^4,5\cdot10^4]$  &  19 &  930 & 2.04\% \\
 $[5\cdot10^4,6\cdot10^4]$  &  16 &  924 & 1.73\% \\
 $[6\cdot10^4,7\cdot10^4]$  &  20 &  878 & 2.28\% \\
 $[7\cdot10^4,8\cdot10^4]$  &  15 &  902 & 1.66\% \\
 $[8\cdot10^4,9\cdot10^4]$  &  15 &  876 & 1.71\% \\
 $[9\cdot10^4,10^5]$        &   6 &  879 & 0.68\% \\
 $[10^5,10^5+10^3]$         &   0 &   81 & 0.00\% \\
 $[10^6,10^6+10^3]$         &   1 &   75 & 1.33\% \\
	\hline\\
	\end{tabular}
}
\caption{Primes with non-empty $\sB_p \cap \sB_p^*$}
\label{table:BpBpstar}
\end{table}

The most surprising result comes from the observation of the intersection
$\sA_p \cap \sB_p$. As  Table~\ref{table:BpAp} shows, the event that this intersection is not empty is rather common.
For any $a \in \sA_p \cap \sB_p$, the graph $\cG_a$ not only has the maximal number of cyclic points but also has a maximal cycle.

Note that for the last two rows we only give primes  in the ranges $[10^5, 10^5+10^3]$
and  $[10^6, 10^6+10^3]$, respectively, due to the limits of our current computational facilities.

\begin{table}[H]
{\small
	\begin{tabular}{lrrc}
	\hline \\[-2.2ex]
  range of $p$ &
  freq &
  \#primes &
  \% \\
	\hline\\[-1.5ex]
  $[3,1\cdot10^4]$           & 268 & 1,228 & 20.36\%\\
  $[10^4,2\cdot10^4]$        & 197 & 1,033 & 18.87\%\\
  $[2\cdot10^4,3\cdot10^4]$  & 153 &  983 & 15.16\%\\
  $[3\cdot10^4,4\cdot10^4]$  & 148 &  958 & 15.24\%\\
  $[4\cdot10^4,5\cdot10^4]$  & 126 &  930 & 13.55\%\\
  $[5\cdot10^4,6\cdot10^4]$  & 167 &  924 & 17.97\%\\
  $[6\cdot10^4,7\cdot10^4]$  & 143 &  878 & 16.17\%\\
  $[7\cdot10^4,8\cdot10^4]$  & 143 &  902 & 15.74\%\\
  $[8\cdot10^4,9\cdot10^4]$  & 144 &  876 & 16.44\%\\
  $[9\cdot10^4,10^5]$        & 147 &  879 & 16.72\%\\
  $[10^5,10^5+10^3]$         &  13 &   81 & 16.05\%\\
  $[10^6,10^6+10^3]$         &   9 &   77 & 11.69\%\\
	\hline\\
	\end{tabular}
}
\caption{Primes with non-empty $\sA_p \cap \sB_p$}
\label{table:BpAp}
\end{table}

\section{Statistics of small cycles}
\label{sect:smallcyclic}

We now study components by analysing the distribution of the size of their cycles.
Let $\cC_{a,k}$ be the number of cycles of length $k$
in the graph $\cG_a$. Let
$$
\cC_k = \sum_{a\in \F_p} \cC_{a,k}
$$
be the number of cycles of length $k$ over all
graphs modulo $p$.
Clearly, we have $\cC_k=0$ for any $k\ge p/2$; see~\cite[Theorems~1 and~2]{PMMY} for better bounds of $k$.

\begin{prop}   \label{prop:Ck}
For any integer $k\ge 1$, there is a constant $D_k$ depending only on $k$ such that for any prime $p >  D_k$ we have
$$
\cC_k = p/k + O\( 4^k k^{-1} p^{1/2}\).
$$
\end{prop}

\begin{proof}
We can assume that $p>k$.
For any fixed $a$, notice that any point $x$ contributing to $\cC_{a,k}$ is a root of the polynomial $F_a^{(k)}(X)$.
Conversely, any root $x$ of $F_a^{(k)}(X)$ contributes to $\cC_{a,d}$ for some $d\mid k$
(possibly $d\ne k$).
Thus, we have
$$
k\cC_k \le \#\{(a,x) \in \F_p^2:~F_a^{(k)}(x) = 0\}.
$$
Moreover, from~\cite[Theorem~2.4~(c)]{MorPa} and noticing $p \nmid k$, we know that if $F_a^{(d)}(x)=0$ and $F_a^{(k)}(x)=0$ with $d<k$,
where $x$ is a point lying in a cycle of length $k$,
then $(X-x)^2 \mid F_a^{(k)}(X)$, that is, the discriminant of $F_a^{(k)}(X)$ is zero.
Note that as a polynomial in $X$ the degree of $F_a^{(k)}(X)$ is at most $2^k$, and as a polynomial in $a$ the degree of $F_a^{(k)}(X)$ is at most $2^{k-1}$.
Then, as a polynomial in $a$, the degree of the discriminant of $F_a^{(k)}(X)$ is at most $4^k$.
Thus, except for at most $4^k$ values of $a$, we have that $F_a^{(k)}(X)$ is a simple polynomial in $X$.
Hence, we have
\begin{equation}   \label{eq:Ck}
k\cC_k = \#\{(a,x) \in \F_p^2:~F_a^{(k)}(x) = 0\} + O(8^k).
\end{equation}

In addition, combining \cite[Corollary~1 to Theorem~B]{Mort} with \cite[Proposition 3.2]{MorVi},
if we view $f_A(X) = X^2 +A $ as an integer polynomial  in variables $A$ and $X$, then $F_A^{(k)}(X) \in \Z[A,X]$ is an absolutely irreducible polynomial.
Then, by Ostrowski's theorem, there exists a positive integer $D_k$
depending only on $k$ such that for any $p > D_k$
 the polynomial $F_A^{(k)}(X)$ is absolutely irreducible
modulo $p$ in variables $A$ and $X$. It is also easy to see by induction on $k$
that $ f_A^{(k)}(X)$ is of total degree  at most $2^k$ as a bivariate polynomial in $A$ and $X$, and the same is true for $F_A^{(k)}(X)$.
Thus, by the Hasse-Weil bound (see~\cite[Section~VIII.5.8]{Lor}) we obtain
$$
\#\{(a,x) \in \F_p^2:~F_a^{(k)}(x) = 0\} = p + O(4^k p^{1/2}), \quad \textrm{as $p \to \infty$},
$$
which, together with~\eqref{eq:Ck}, implies the desired result (as we can always assume that
$D_k > 4^k$, so $4^k p^{1/2}> 8^k$).
\end{proof}

In particular, we see from  Proposition~\ref{prop:Ck} that for any fixed integer $k\ge 1$,
$$
\cC_k \sim p/k, \qquad \text{as} \ p \to \infty.
$$

Note that using~\cite[Theorem~1]{GaoRod} or~\cite[Satz~B]{Rup} or~\cite[Corollary]{Zannier}, one can obtain an explicit form for $D_k$. However, any such estimate has to depend on the size of the
coefficients of $F_A^{(k)}(X)$ (considered as a bivariate polynomial in $A$ and $X$ over $\Z$) and
is likely to be double exponential in $k$.

We can also compute the exact values of $\cC_1$ and $\cC_2$.

\begin{prop}   \label{prop:cycle}
For any odd prime $p$, we have $\cC_1=p$ and $\cC_2= (p-1)/2$.
\end{prop}

\begin{proof}
First, note that any point $x$ contributing to $\cC_1$ is a root of $F_a^{(1)}(X)$ for some $a$, and also
$$
F_a^{(1)}(X) = X^2 - X + a=(X-1/2)^2 + a -1/4=0
$$
is solvable if and only if $1/4-a$ is a square.
Since there are $(p-1)/2$ squares in $\F_p^*$, we have $\cC_1 = p$.

Now, it is easy to see that
$$
F_a^{(2)}(X) = X^2 + X + a + 1.
$$
If a point $x$ lies in a cycle of length $2$ in $\cG_a$, then it is a root of $F_a^{(2)}(X)$ and also it is not a root of $F_a^{(1)}(X)$.
However, if there exists a point $x$ such that
$$
F_a^{(2)}(x)=F_a^{(1)}(x)=0,
$$
then we must have $x=-1/2, a=-3/4$.
So, if $a\ne -3/4$, then any root of $F_a^{(2)}(X)$ lies in a cycle of length $2$.
Thus, noticing that
$$
F_a^{(2)}(X) =(X+1/2)^2 + a + 3/4=0
$$
is solvable if and only if $-a-3/4$ is a square, we have $\cC_2=(p-1)/2$ and conclude
the proof.
\end{proof}

Table~\ref{tab:cyclic_point_dist} shows the $\cC_k$ for some
values of $p$ (in these cases, we also included the graphs $X^2$ and $X^2-2$).
This is consistent with Proposition~\ref{prop:Ck}.

\begin{table}[H]
{\small
	\begin{tabular}{r r r r r r r}
	\hline \\[-2.2ex]
	\multicolumn{1}{c}{$k$} &
	\multicolumn{2}{c}{$p = 100,003$} &
	\multicolumn{2}{c}{$p = 500,009$} &
	\multicolumn{2}{c}{$p = 1,000,003$} \\
  & $\cC_k$ & \multicolumn{1}{c}{$\fl{p/k}$}
  & $\cC_k$ & \multicolumn{1}{c}{$\fl{p/k}$}
  & $\cC_k$ & \multicolumn{1}{c}{$\fl{p/k}$}  \\
	\hline\\[-1.5ex]
    1 & 100,003 & 100,003 & 500,009 & 500,009 & 1,000,003 & 1,000,003  \\
    2 &  50,001 &  50,001 & 250,004 & 250,004 &  500,001 &  500,001  \\
    3 &  33,333 &  33,334 & 166,669 & 166,669 &  333,333 &  333,334  \\
    4 &  24,890 &  25,000 & 125,000 & 125,002 &  249,890 &  250,000  \\
    5 &  20,061 &  20,000 &  99,353 & 100,001 &  199,310 &  200,000  \\
    6 &  16,775 &  16,667 &  83,664 &  83,334 &  165,852 &  166,667  \\
    7 &  14,179 &  14,286 &  71,582 &  71,429 &  143,109 &  142,857  \\
    8 &  12,474 &  12,500 &  62,541 &  62,501 &  125,266 &  125,000  \\
	\hline\\
	\end{tabular}
}
\caption{Number of cycles of length $k$}
\label{tab:cyclic_point_dist}
\end{table}

\section{Distribution of components with size $k$}
\label{sect:smallsize}

We now study the components of functional graphs by analysing the distribution of their sizes.
For the minimal and maximal numbers of components in graphs $\cG_a$ as well as the popular component size, we refer to~\cite[Sections~4.4 and~4.5]{KLMMSS}.

Let ${\mathcal N}_p$ be the number of components taken over all ${\mathcal G}_a$ modulo $p$,
and let ${\mathcal N}_{p,k}$ be the number of those components with size $k > 0$
(that is, there are $k$ nodes in the component).
Furthermore, let
$$
  {\mathcal N}_{p,\text{even}}^K = \sum_{\substack{k \le K\\ \text{$k$ even}}} {\mathcal N}_{p,k}
  \quad
  \text{and}
  \quad
  {\mathcal N}_{p,\text{odd}}^K  = \sum_{\substack{k \le K\\ \text{$k$ odd}}} {\mathcal N}_{p,k}.
$$
Clearly,
$$
\cN_{p} = \cN_{p,\text{even}}^p + \cN_{p,\text{odd}}^p.
$$

We first have:

\begin{prop}
For any odd prime $p$,  $\cN_{p,2}=(p-1)/2$.
\end{prop}

\begin{proof}
If $C$ is a component of $\cG_a$ of size $2$,
then it is easy to see that $C=\{x,-x\}$ for some $x\in \F_p$ such that
$x$ is a fixed point (that is, $x^2+a=x$) and the equation $X^2+a=-x$ has no solution in $\F_p$ (that is, $-x-a$ is not a square).

In other words, for any $x\in \F_p$, if we choose $a=-x^2+x$, then $x$ is a fixed point in $\cG_a$
and $-x-a=x^2-2x$.
So, it is equivalent to count how many $x\in \F_p$ such that $x^2-2x$ is not a square in $\F_p$.
Since $x^2-2x=(x-1)^2-1$, it is also equivalent to count how many $x\in \F_p$ such that $x^2-1$ is not a square in $\F_p$.

If $x^2-1$ is a square in $\F_p$, say $x^2-1=y^2$, then we have $(x+y)(x-y)=1$.
Let $\alpha=x+y$, then $x-y=\alpha^{-1}$, and so
$$
x = \frac{\alpha+\alpha^{-1}}{2},  \qquad     y = \frac{\alpha-\alpha^{-1}}{2}.
$$
So, for such pairs $(x,y)$ we obtain a one-to-one correspondence between pairs $(x,y)$ and pairs $(\alpha,\alpha^{-1}),\alpha \ne 0$.
It is easy to see that for any $\alpha_1,\alpha_2\in \F_p^*$,
$$
\textrm{$\frac{\alpha_1+\alpha_1^{-1}}{2}=\frac{\alpha_2+\alpha_2^{-1}}{2}$ if and only if $\alpha_1\alpha_2=1$.}
$$
So, by counting the pairs $(\alpha,\alpha^{-1})$, there are $(p+1)/2$ values of $x$ such that $x^2-1$ is a square.
Therefore,  there are $(p-1)/2$ values of $x$ such that $x^2-1$ is not a square.
This completes the proof.
\end{proof}

It has been predicted in~\cite[Theorem~2~(i)]{FO2} that
\[
  {\mathcal N}_p \sim \frac{p\log p}{2},
\]
which has a small bias (about $9.5\%$) over the real value; see~\cite[Table~4.2]{KLMMSS}.
Here, we improve the precision of this estimate.
First, we note that each node in ${\mathcal G}_a$ has in-degree two or zero
except for the node $a$, since only $0$ maps to $a$.
Therefore, each component in any graph ${\mathcal G}_a$ has an even number of nodes unless it is the
component containing $0$ and $a$.
So, each graph ${\mathcal G}_a$ has exactly one component of odd size.
It follows that
$$
{\mathcal N}_{p,\text{odd}}^p = p,
$$
and so
$$
\cN_p \sim \cN_{p,\text{even}}^p, \quad \textrm{as $p \to \infty$}.
$$

For even-sized components, the situation is not as straightforward.
In our experiments, we noticed that the number of even-sized components with
size $k$ is very close to $p/k$ as shown in Table~\ref{tab:component_size_dist} for $k \le 20$ and for
$k = 1000$ and $2000$ (i.e., even for larger values of $k$).

\begin{table}[H]
{\small
	\begin{tabular}{r r r r r r r}
	\hline \\[-2.2ex]
	\multicolumn{1}{r}{$k$} &
	\multicolumn{2}{c}{$p =  100,003$} &
	\multicolumn{2}{c}{$p =  500,009$} &
	\multicolumn{2}{c}{$p = 1,000,003$} \\
  & ${\mathcal N}_{p,k}$ & \multicolumn{1}{c}{$\fl{p/k}$}
  & ${\mathcal N}_{p,k}$ & \multicolumn{1}{c}{$\fl{p/k}$}
  & ${\mathcal N}_{p,k}$ & \multicolumn{1}{c}{$\fl{p/k}$}  \\
	\hline\\[-1.5ex]
     2 & 50,001 & 50,001 & 250,004 & 250,004  & 500,001 & 500,001  \\
     4 & 24,951 & 25,000 & 125,160 & 125,002  & 250,171 & 250,000  \\
     6 & 16,156 & 16,667 &  83,185 &  83,334  & 166,660 & 166,667  \\
     8 & 12,509 & 12,500 &  62,652 &  62,501  & 124,727 & 125,000  \\
    10 & 10,083 & 10,000 &  50,422 &  50,000  &  99,975 & 100,000  \\
    12 &  8,389 &  8,333 &  41,542 &  41,667  &  82,577 &  83,333  \\
    14 &  7,192 &  7,143 &  35,661 &  35,714  &  71,611 &  71,428  \\
    16 &  6,292 &  6,250 &  31,186 &  31,350  &  62,220 &  62,500  \\
    18 &  5,503 &  5,555 &  27,941 &  27,778  &  55,923 &  55,555  \\
    20 &  5,009 &  5,000 &  24,662 &  25,000  &  50,135 &  50,000  \\
  1000 &   117 &   100 &    533 &    500  &    954 &   1,000  \\
  2000 &    48 &    50 &    243 &    250  &    489 &    500  \\
	\hline\\
	\end{tabular}
}
\caption{Number of components of size $k$}
\label{tab:component_size_dist}
\end{table}

Now, using $\fl{p/k}$ as an approximation of the number of components of size $k$ for any even $k<p$,
we can get an approximation for $\cN_{p,\text{even}}^p$.
First, when $(p-1)/2 < k < p $, we have $\fl{p/k}=1$, and there are about $(p-1)/4$ values of such even $k$.
In general, if $(p-1)/(n+1) < k \le (p-1)/n $, we have $\fl{p/k}=n$, and there are about $\frac{p-1}{2n(n+1)}$ values of such even $k$,
which contributes to around $\frac{p-1}{2(n+1)}$ components of even size.

Fixing a positive integer $n$, for $k>(p-1)/(n+1)$ we use the above estimate, while for $k\le (p-1)/(n+1)$ we use the estimate $(p-1)/k$,
and so the total number of components of even size is around
\begin{align*}
\frac{p-1}{2}&\(1+\frac{1}{2} + \cdots + \frac{1}{(p-1)/(2(n+1))}\) \\
& \qquad \qquad + \frac{p-1}{2} \(\frac{1}{2} + \frac{1}{3} + \cdots + \frac{1}{n+1}\),
\end{align*}
which, together with the approximation of the harmonic series, is approximated by
\begin{align*}
\frac{p-1}{2}  \(\log\frac{p-1}{2(n+1)}+\gamma\) + \frac{p-1}{2}\(-1+\log(n+1)+\gamma\) \qquad &\\
 = \frac{p-1}{2}\( \log(p-1) +2\gamma -1 - \log 2 \)&,
\end{align*}
where $\gamma = 0.5772156649\dots$ is the Euler constant.
So, we denote
$$
\widetilde{\mathcal N}_{p,\text{even}}^p = \frac{p-1}{2}\( \log(p-1) +2\gamma -1 - \log 2 \),
$$
which is an approximation of ${\mathcal N}_{p,\text{even}}^p$.

Table~\ref{tab:components_small_size} shows the difference between the two values for several large primes.
We overestimate the actual value by about 2\%.

\begin{table}[H]
{\small
	\begin{tabular}{r c c c c}
	\hline \\[-2.2ex]
  \multicolumn{1}{c}{$p$} &
  ${\mathcal N}_{p,\text{even}}^K$ &
  ${\mathcal N}_{p,\text{even}}^p$ &
  $\cN_p$ &
  $\widetilde{\mathcal N}_{p,\text{even}}^p$ \\
	\hline\\[-1.5ex]
  100,003 &  521,337 &  538,640 &   638,643 &548,722\\
  200,003 & 1,113,083 & 1,147,694 &  1,347,697 &1,166,748 \\
  300,007 & 1,730,420 & 1,782,805 &  2,082,812 &1,810,962 \\
  400,009 & 2,364,734 & 2,434,894 &   2,834,903 &2,472,154\\
  500,009 & 3,011,626 & 3,098,914 &  3,598,923 &3,145,966\\
  600,011 & 3,667,637 & 3,772,277 &  4,372,288 &3,829,859\\
  700,001 & 4,333,622 & 4,455,913 &  5,155,914 &4,522,041\\
  800,011 & 5,005,995 & 5,145,194 &  5,945,205 &5,221,530\\
  900,001 & 5,685,731 & 5,842,337 &  6,742,338 &5,927,145\\
 1,000,003 & 6,369,257 & 6,543,317 & 7,543,320  &6,638,411\\
	\hline\\
	\end{tabular}
}
\caption{Estimates for the number of components with even size and $K = (p-1)/2$}
\label{tab:components_small_size}
\end{table}

\section{Shape of trees in functional graphs}
\label{sect:tree}

Finally, in order to reveal more detailed features of functional graphs, we consider the trees attached to such graphs.

In the functional graph $\cG_a$ corresponding to $f_a$, each node in a cycle, except for $a$ (if $a$ lies in a cycle), is connected to a unique node (say $w$)
which is not in the cycle.
Naturally, we treat the node $w$ as the root of the binary tree attached to a cyclic point in the graph $\cG_a$.
Thus, we can say that each node in a cycle of $\cG_a$, expect for $a$, is associated with a binary tree --
in fact a full binary tree, unless $0$ is a node in the tree.
For example, in Figure~\ref{pic:connected_graph}, there are $8$ full binary trees attached to the cyclic points.
Let $t_p(a,k)$ be the number of such binary trees with $k$ nodes in $\cG_a$, and let
\[
  T_{p}(k) = \sum_{a \in \F_p} t_p(a,k) \quad \text{and} \quad
  T_p = \sum_{k=1}^{p-1} T_p(k);
\]
and for the connected graphs equivalents, let
\[
  T^*_{p}(k) = \sum_{a \in \cI_p} t_p(a,k) \quad \text{and} \quad
  T_p^* = \sum_{k=1}^{p-1} T_p^*(k).
\]
Note that $T_p$ is the total number of trees attached to all such functional graphs $\cG_a$,  
and $T_p^*$ has a similar meaning but with restriction to connected functional graphs. 

An interesting question is whether these trees behave similarly to random full binary trees.
First we observe that there is a significant proportion of trees with just one node, as shown
in Table~\ref{table:tree_numbers} for the general case and
in Table~\ref{table:tree_numbers_connected} for connected graphs.
This motivates us to pose the following conjecture, which seems to be reasonable because exactly
half of elements in $\F_p^*$ are not square.

\begin{conj}
We have $T_p(1)/T_p \sim 1/2$ as $p \to \infty$.
\end{conj}

\begin{table}[H]
{\small
	\begin{tabular}{r r r r}
	\hline \\[-2.2ex]
	\multicolumn{1}{c}{$p$} &
	\multicolumn{1}{c}{$T_p(1)$} &
	\multicolumn{1}{c}{$T_p$} &
	\multicolumn{1}{c}{\%} \\
	\hline\\[-1.5ex]
   50,111 &   7,090,084 &    14,091,820 & 50.31\% \\
  100,003 &  19,845,915 &    39,530,737 & 50.20\% \\  %
  200,003 &  56,210,936 &   112,088,213 & 50.15\% \\  
  300,007 & 103,203,596 &   205,901,181 & 50.12\% \\ 
  400,009 & 158,746,944 &   317,089,081 & 50.06\% \\  
  500,009 & 221,941,725 &   443,336,032 & 50.06\% \\  
1,000,003 & 627,460,216 & 1,253,326,817 & 50.06\% \\ 
	\hline\\
	\end{tabular}
  }
  \caption{Number of trees with one node}
 \label{table:tree_numbers}
  \end{table}

  \begin{table}[H]
  {\small
  	\begin{tabular}{r r r r}
  	\hline \\[-2.2ex]
  	\multicolumn{1}{c}{$p$} &
  	\multicolumn{1}{c}{$T_p^*(1)$} &
  	\multicolumn{1}{c}{$T_p^*$} &
  	\multicolumn{1}{c}{\%} \\
  	\hline\\[-1.5ex]
       50,111 &    27,877 &    55,668 & 50.08\% \\
      100,003 &    52,923 &   105,612 & 50.11\% \\  %
      200,003 &   115,746 &   231,583 & 49.98\% \\  %
      300,007 &   161,975 &   323,410 & 50.08\% \\ %
      400,009 &   222,865 &   445,931 & 49.98\% \\  %
      500,009 &   298,060 &   595,142 & 50.08\% \\  %
    1,000,003 &   542,592 & 1,086,147 & 49.96\% \\ %
  	\hline\\
  	\end{tabular}
    }
    \caption{Number of trees with one node in connected graphs}
   \label{table:tree_numbers_connected}
    \end{table}

Second, for large trees, we check the average height of the trees in the graphs.
It has been shown in~\cite[Theorem~B]{FO} that the average height of full binary trees with $n$ internal nodes is
\[
	\overline{H}_n \sim 2\sqrt{\pi n} \qquad \text{as $n \rightarrow \infty$.}
\]
This means that for a random full binary tree, its height is asymptotic to $2\sqrt{\pi n}$ when
$n$ goes to the infinity.
In our situation, for each tree with $n$ internal nodes and height $H_n$, we compute the ratio
$H_n/2\sqrt{\pi n}$ and find the average of this ratio for all graphs modulo $p$.
(Again, a tree is not always guaranteed to be a full binary tree, since $0$ might be a node
in the tree, but the impact of this happening is negligible,
and at any case, we collect trees of both sizes $2n$ and $2n+1$.)

In Table~\ref{table:tree_heights}, we compare the  ratio of $\overline{H}_n/2\sqrt{\pi n}$
(see~\cite[Table~II]{FO}) with the average
ratio of $H_n/2\sqrt{\pi n}$ of the trees in our graphs.
One can see that they are close.

\begin{table}[H]
{\small
	\begin{tabular}{r r r r r}
	\hline \\[-2.2ex]
	\multicolumn{1}{c}{$n$} &
	\multicolumn{1}{c}{$\overline{H}_n/2\sqrt{\pi n}$} &
	\multicolumn{3}{c}{average of $H_n/2\sqrt{\pi n}$} \\
	& & $p = 50111$ & $p=100003$ & $p=200003$\\
	\hline\\[-1.5ex]
		50		& 0.797 & 0.837  & 0.837  & 0.837  \\
		100   & 0.846 & 0.875  & 0.873  & 0.872\\
		500   & 0.920 & 0.952  & 0.925  & 0.941 \\
		1,000	& 0.940 & 0.925  & 0.948  & 0.942 \\
		2,000  & 0.956 & 0.981  & 0.944  & 0.960\\
		5,000  & 0.970 & 0.927  & 0.916  & 0.977 \\
	\hline\\
	\end{tabular}
}
\caption{Average height of trees}
\label{table:tree_heights}
\end{table}

\section{Future Directions}

One of the most important   directions in this area is  developing an adequate
random model predicting the statistical characteristics of the functional graphs of polynomials,
see~\cite{MaPa} for some initial, yet promising
results in this direction.

Based on our computations, we pose several conjectures about the functional graphs of quadratic polynomials.
Investigating whether they are true or not may help to characterise functional graphs generated by quadratic polynomials and  understand the similarities and differences between these functional graphs and random mappings.

The other interesting problem is to count the number of functional graphs modulo $p$ generated by quadratic polynomials up to isomorphism;
see~\cite[Theorem~2.8]{KLMMSS} for a lower bound.
In~\cite[Conjecture~C]{GKRS} the authors conjectured that for any odd prime $p\ne 17$, there are $p$ such functional graphs up to isomorphism,
and they confirmed this for all the odd primes up to 1009 not equal to $17$. Under our computations, we confirm this conjecture
for all the odd primes up to 100000 not equal to $17$.

\section*{Acknowledgements}

The authors are grateful to Patric Morton and Michael Zieve for several useful suggestions and
literature references,  especially concerning dynatomic polynomials.

For the research, B.M. was partially supported
by the  Australian Research Council Grants DP140100118 and DP170102794,  M.S. by the
Macquarie University Research Fellowship, I.S. by the
Australian Research Council Grants~DP130100237  and DP140100118.

\end{document}